\documentclass[reqno,a4paper,11pt]{amsart}

\usepackage{enumitem}

\setlist[enumerate, 1]{label = \roman*., ref = \roman*}
\setlist[enumerate, 2]{label = \theenumi.\alph*}

\usepackage{xcolor, comment}
\usepackage{float}
\usepackage{amsmath, amscd, amssymb, amsfonts, euscript, mathtools,stmaryrd}
\usepackage{amsthm}
\numberwithin{equation}{section}
\usepackage{graphicx}

\usepackage{diagbox}
\usepackage{epigraph}

\usepackage[noadjust]{cite}
\usepackage[
pagebackref=true,
colorlinks=true,
urlcolor=purple,
linkcolor=purple!87!black,
citecolor=green!60!black,
pdfborder={0 0 0}
]{hyperref}
\renewcommand*{\backref}[1]{}
\renewcommand*{\backrefalt}[4]{[{\tiny%
		\ifcase #1 Not cited.%
		\or Cited on page~#2.%
		\else Cited on pages #2.%
		\fi%
	}]}
\usepackage{cleveref}

\usepackage{url}
\hypersetup{breaklinks=true}
\usepackage{geometry}
\DeclarePairedDelimiter\norm{\lVert}{\rVert}
\DeclarePairedDelimiter\abs{\lvert}{\rvert}
\makeatletter
\let\oldabs\abs
\def\abs{\@ifstar{\oldabs}{\oldabs*}}
\let\oldnorm\norm
\def\norm{\@ifstar{\oldnorm}{\oldnorm*}}
\makeatother

\newcommand{\dd}{\mathop{}\!\mathrm{d}}

\mathtoolsset{centercolon}

\def\multiset#1#2{\ensuremath{\left(\kern-.3em\left(\genfrac{}{}{0pt}{}{#1}{#2}\right)\kern-.3em\right)}}

\newcommand\restr[2]{{
		\left.\kern-\nulldelimiterspace 
		#1 
		\vphantom{\big|} 
		\right|_{#2} 
}}

\def\C{{\mathbb C}}

\def\N{{\mathbb N}}

\def\P{{\mathbb P}}

\def\R{{\mathbb R}}

\def\Z{{\mathbb Z}}

\def\cO{{\mathcal O}}

\def\cT{{\mathcal T}}

\newcommand{\sub}[1]{\mathrm{Sub}(#1)}
\newcommand{\subam}[1]{\mathrm{Sub}_{\mathrm{am}}(#1)}

\let\epsilon\varepsilon
\let\emptyset\varnothing

\newtheorem{teo}{Theorem}[section]

\newtheorem{lema}[teo]{Lemma}
\newtheorem{defn}[teo]{Definition}
\newtheorem{prop}[teo]{Proposition}

\newtheorem{cor}[teo]{Corollary}

\newtheorem*{claim}{Claim}

\newcommand{\thistheoremname}{}

\newtheorem*{genericthm*}{\thistheoremname}
\newenvironment{namedthm*}[1]
{\renewcommand{\thistheoremname}{#1}%
	\begin{genericthm*}}
	{\end{genericthm*}}

\newtheorem{thmA}{Theorem}

\newtheorem{corA}[thmA]{Corollary}

\theoremstyle{definition}

\newtheorem*{ejm*}{Example}

\newtheorem{rmk}[teo]{Remark}
\newtheorem*{rmks}{Remarks}

\newtheorem{qnnum}{Question}

\setlist[enumerate, 1]{label = \roman*., ref = \roman*}
\setlist[enumerate, 2]{label = \theenumi.\alph*}

\usepackage{hyperref}
\usepackage{url}
\hypersetup{breaklinks=true}

\author{Anna Cascioli}
\email[Anna Cascioli]{acasciol@uni-muenster.de}
\address{University of Münster, Einsteinstrasse 62, Münster 48149, Germany}
\urladdr{\url{https://annacascioli.github.io/}}
\author{Martín Gilabert Vio}
\email[Martín Gilabert Vio]{gilabert@math.univ-lyon1.fr, martingilabertvio@gmail.com}
\address{Institut Camille Jordan, Université Claude Bernard Lyon 1, 43 boulevard du 11 novembre
	1918, 69622 Villeurbanne, France}
\urladdr{\url{https://martingilabertvio.github.io/}}
\author{Eduardo Silva}
\email[Eduardo Silva]{eduardo.silva@uni-muenster.de, edosilvamuller@gmail.com}
\address{University of Münster, Einsteinstrasse 62, Münster 48149, Germany}
\urladdr{\url{https://edoasd.github.io/eduardo_silva_math/}}

\title{Stationary boundaries on the space of amenable subgroups and C*-simplicity}

\begin{document}
	
	\begin{abstract}
		We give a sufficient condition for a countable group $G$ to possess a probability measure $\mu$ that admits a non-trivial $\mu$-boundary modeled in the space $\mathrm{Sub}_{\mathrm{am}}(G)$ of amenable subgroups of $G$. In particular, for such $\mu$ the space $\mathrm{Sub}_{\mathrm{am}}(G)$ is not uniquely $\mu$-stationary. This contrasts with a theorem of Hartman-Kalantar, which states that a countable group $G$ is C*-simple if and only if there exists $\mu\in \mathrm{Prob}(G)$ such that $\mathrm{Sub}_{\mathrm{am}}(G)$ is uniquely $\mu$-stationary \cite{HartmanKalantar2022}. Our criterion applies to (permutational) wreath products, which include groups that are C*-simple, and to Thompson's group $F$, whose C*-simplicity is equivalent to its non-amenability and therefore remains an open problem. We also show that any non-trivial $\mu$-boundary modeled on $\mathrm{Sub}_{\mathrm{am}}(G)$ is supported on amenable normalish subgroups, in the sense of Breuillard-Kalantar-Kennedy-Ozawa \cite{BreuillardKalantarKennedyOzawa2017}. As a consequence, we conclude that a countable group with no finite normal subgroups and no amenable normalish subgroups acts essentially freely on all its Poisson boundaries.
		
	\end{abstract}
	
	\maketitle
	
\section{Introduction}
The space $\sub{G}$ of subgroups of a discrete countable group $G$ is naturally identified with a closed subset of $\{0,1\}^G$ endowed with the product topology. The induced topology, known as the \emph{Chabauty topology}, makes $\sub{G}$ a compact metrizable space, and the conjugation action of $G$ on its subgroups induces an action by homeomorphisms $G \curvearrowright \sub{G}$. A central object in the study of this dynamical system is the class of $G$-invariant probability measures on $\sub{G}$, known as \emph{invariant random subgroups} (IRSs). Notable applications of IRSs include rigidity results such as the Stuck--Zimmer theorem \cite{StuckZimmer}, approximation results for the $L^2$-Betti numbers of semisimple Lie groups \cite{7samurai}, and the realization of entropy spectra of stationary actions of groups \cite{Bowen2015,HartmanYadin2018}.

The set $\mathrm{Sub}_\mathrm{am}(G)$ of amenable subgroups of $G$ is closed in $\mathrm{Sub}(G)$, and a reason to study the dynamics of $G \curvearrowright \subam{G}$ is its connection with the C*-simplicity of $G$. A countable group $G$ is said to be \emph{C*-simple} if its reduced C*-algebra $C^\ast_r(G)$ contains no non-trivial closed ideals, where the \emph{reduced C*-algebra $C_r^\ast(G)$ of $G$} is the norm closure of the linear span of the unitary operators defined by the left regular representation of $G$ on $\ell^2(G)$. This property has been well studied from the perspective of operator algebras and group theory \cite{Powers1975, BekkaCowlingDeLaHarpe1994}. A result of M. Kennedy \cite{Kennedy2020}, following his work with Breuillard-Kalantar-Ozawa \cite{KalantarKennedy2017, BreuillardKalantarKennedyOzawa2017}, states that a countable group $G$ is C*~-simple if and only if the only minimal closed invariant subset of $\subam{G}$ is the one composed of the trivial subgroup $\{e_G\}$.

The statement of a measurable version of the previous criterion involves stationary probability measures rather than invariant ones. Recall that a probability measure $\mu \in \mathrm{Prob}(G)$ is \emph{non-degenerate} if the semigroup generated by its support is all of $G$, and that a measurable action of $G$ on a probability space $(Z, \eta)$ is \emph{$\mu$-stationary} if $\eta = \sum_{g \in G} \mu(g)g_\ast \eta$. A $\mu$-stationary probability measure on $\subam{G}$ is called an \emph{amenable $\mu$-stationary random subgroup} (or \emph{amenable $\mu$-SRS)}. This notion turns out to be useful to characterize C*-simplicity of countable groups by work of Hartman-Kalantar \cite{HartmanKalantar2022}, who proved that a countable group $G$ is C*-simple if and only if there exists a non-degenerate probability measure $\mu \in \mathrm{Prob}(G)$ such that $\subam{G}$ admits a unique $\mu$-stationary probability measure (which must be $\delta_{\{e_G\}}$). In other words, C*-simple groups possess distinguished probability measures that witness their C*-simplicity. In this paper we address the existence of probability measures with the opposite property.

\begin{qnnum}\label{question: main question}
	Let $G$ be a countable group. Does there exist a non-degenerate probability measure $\mu$ on $G$ that admits an amenable $\mu$-SRS distinct from $\delta_{\{e_G\}}$?
\end{qnnum}
As observed in \cite{HartmanKalantar2022}, if $G$ is not C*-simple then every non-degenerate probability measure $\mu$ on $G$ has the above property. It follows that Question~\ref{question: main question} is of interest only when $G$ is C*-simple. Moreover, Hartman-Kalantar show that there are classes of C*-simple groups for which no probability measure $\mu$ admits an amenable $\mu$-SRS distinct from $\delta_{\{e_G\}}$. This holds for hyperbolic groups, mapping class groups and linear groups \cite[Theorems 4.12, 6.5 \& 6.7, Example 6.6]{HartmanKalantar2022}, assuming that their amenable radical (that is, their maximal normal amenable subgroup) is trivial. This discussion suggests a graded notion of C*-simplicity, where groups admitting no non-trivial amenable $\mu$-SRSs can be viewed as “more” C*-simple, in the sense that their stationary dynamics on $\mathrm{Sub}_\mathrm{am}(G)$ are highly constrained.

Our results deal with the following class of amenable SRSs.

\begin{defn}\label{def: boundary SRS}
	Let $\mu$ be a probability measure on a countable group $G$ and denote by $\P_{\mu}$ the law of the right $\mu$-random walk on $G$. A $\mu$-stationary probability measure $\eta$ on the space $\subam{G}$ of amenable subgroups of $G$ is called an \emph{amenable boundary $\mu$-stationary random subgroup} if, for $\P_{\mu}$-almost every trajectory $(w_n)_{n\ge 0}\in G^\N$, the sequence $\left(w_{n \ast} \eta\right)_{n\ge 0}$ converges in the weak*-topology to a Dirac mass in $\mathrm{Sub}_\mathrm{am}(G)$.
\end{defn}

Equivalently,  $\eta$ is an amenable boundary $\mu$-SRS if the space $(\subam{G},\eta)$ is a $\mu$-boundary, in the sense that it is a $G$-equivariant quotient of the Poisson boundary of $(G,\mu)$. Our first main theorem provides a condition that guarantees the existence of such a measure.

\begin{thmA}\label{thm: syndetic}
	Let $G$ be a countable group. Suppose that there exists a non-trivial subgroup $H$ such that for all finite subsets $Q, Z \subseteq G$ there is $b \in G$ such that both $bZ$ and $b^{-1}Z$ are contained in
	\[
	\left\{g \in G : Q \cap H = Q \cap gHg^{-1}\right\}.
	\]
	Then there exists a non-degenerate, symmetric and finite-entropy probability measure $\mu$ on $G$ such that $\overline{\mathrm{Orb}_G(H)}$ supports a $\mu$-boundary SRS distinct from $\delta_{\{e_G\}}$.
	
	In particular, if $H$ is amenable then $G$ admits an amenable boundary $\mu$-SRS distinct from $\delta_{\{e_G\}}$.
\end{thmA}
The condition appearing in the statement can be viewed as a strong form of recurrence for the action of $G$ on the orbit of $H$ (see Remark \ref{remark: dynamical interpretation}). The strategy we use to construct $\mu$ is the method of record times used by Frisch-Hartman-Tamuz-Vahidi Ferdowsi \cite{FrischHartmanTamuzVahidiFerdowsi2019} to prove that any non-hyper-FC-central group admits non-degenerate probability measures with a non-trivial Poisson boundary.

In Proposition~\ref{prop: wreath products}, we also present a shorter and more direct argument for wreath products, which differs from the general proof of Theorem~\ref{thm: syndetic}.  Recall that if $A$ is a C*-simple group, then for every countable group $B$ the wreath product $A \wr B=\bigoplus_B A\rtimes B$ is C*-simple. Proposition~\ref{prop: wreath products} provides examples of C*-simple groups for which any probability measure witnessing C*-simplicity, as in the main theorems of \cite{HartmanKalantar2022} must necessarily have an infinite support (see Corollary \ref{cor: csimple infinite first moment}). To the best of our knowledge, these are the first examples of this kind. A. Erschler and J. Frisch communicated to us that they independently obtained Proposition \ref{prop: wreath products}.

The criterion in Theorem \ref{thm: syndetic} applies to Thompson's group $F$, the group of piecewise dyadically affine homeomorphisms of the interval.

\begin{corA} \label{cor: thompson F}
	There exists a non-degenerate, symmetric and finite-entropy probability measure $\mu$ on Thompson's group $F$ such that $\mathrm{Sub}_\mathrm{ab}(F) = \{H \in \mathrm{Sub}(F) : H \text{ is abelian}\}$ supports a non-atomic $\mu$-boundary.
\end{corA}

\begin{rmks}
	\ 
	
	\begin{itemize}
		\item The minimal closed subsystems of $\mathrm{Sub}(F)$ and the IRSs of $F$ are supported on normal subgroups (see \cite[Theorem 1.7, (i)]{LeBoudecMatteBon2018} and \cite{DudkoMedynets2014} respectively), which are the subgroups of $F$ containing the (non-abelian) derived subgroup $[F,F]$. Corollary~\ref{cor: thompson F} shows that the dynamics of $\subam{F}$ are nevertheless rich enough to support non-trivial $\mu$-SRSs. A well-known open question in group theory asks whether $F$ is amenable \cite{CannonFloydParry1996}, which is equivalent to $F$ not being C*-simple by \cite[Theorem 1.6]{LeBoudecMatteBon2018}. If $F$ were non-amenable, it would follow that $F$ belongs to a class of C*-simple groups whose amenable subgroup dynamics are richer than those of hyperbolic groups or mapping class groups.
		\item The proof of Corollary \ref{cor: thompson F} applies verbatim to the finitely presented non-amenable group of piecewise projective homeomorphisms of $\mathbb{R}$ constructed by Lodha-Moore \cite{LodhaMoore2016}, which is known to be C*-simple \cite[Theorem 1.10]{LeBoudecMatteBon2018}.
	\end{itemize}
\end{rmks}

Recall that a subgroup $H$ of $G$ is called \emph{normalish} if $\bigcap_{z \in Z} zHz^{-1}$ is infinite for every finite subset $Z \subseteq G$. This notion was introduced in \cite[Section 6]{BreuillardKalantarKennedyOzawa2017}, where the absence of normalish subgroups is shown to imply C*-simplicity. Our second main theorem connects this concept to Question \ref{question: main question}.

\begin{thmA}\label{thm: normalish}
	Let $\mu$ be a non-degenerate probability measure on a countable group $G$ and let $\eta$ be a $\mu$-stationary probability measure on $\mathrm{Sub}(G)$. Suppose that $(\mathrm{Sub}(G),\eta)$ is a $\mu$-boundary of $G$, and that $\eta$ is not a Dirac mass on a finite normal subgroup of $G$. Then $\eta$ is supported on normalish subgroups of $G$. 
\end{thmA}

As a consequence, we obtain the next result.
\begin{corA}\label{cor: essentially free}
	Let $G$ be a countable group with no finite normal subgroups and no amenable normalish subgroups. Then for every non-degenerate $\mu \in \mathrm{Prob}(G)$, the action of $G$ on the Poisson boundary of $(G, \mu)$ is essentially free.
\end{corA}

Indeed, the action of a countable group G on its Poisson boundary is amenable \cite[Corollary 5.3]{Zimmer1978}; see also \cite[Theorem 5.2]{Kaimanovich2005} for a probabilistic proof. Elliot-Giordano showed that the point stabilizers of an amenable action are almost surely amenable \cite[Proposition 1.2]{ElliotGiordano1993}, so that the stabilizer map $z \mapsto \mathrm{Stab}(z)$ from the Poisson boundary of $(G, \mu)$ to $\mathrm{Sub}(G)$ takes values in $\mathrm{Sub}_\mathrm{am}(G)$ (see also \cite[Corollary 5.3.33]{AR2000}). The stabilizer map thus defines an amenable boundary $\mu$-SRS on $G$. If the action of $G$ on the Poisson boundary of $(G, \mu)$ were not essentially free, Theorem \ref{thm: normalish} would show the existence of an amenable normalish subgroup of $G$, contradicting the hypotheses of the corollary.

Examples of groups with no amenable normalish subgroups include groups with a positive $L^2$-Betti number \cite[Corollary 1.5]{BaderFurmanSauer2014}, groups with some non-trivial bounded cohomology with mixing coefficients, and linear groups with a finite amenable radical \cite[Propositions 6.3 \& 6.4]{BreuillardKalantarKennedyOzawa2017}. These hold for
\begin{itemize}
	\item  countable groups admitting non-elementary, isometric and proper actions on simplicial trees, proper $\mathrm{CAT}(-1)$ spaces or Gromov-hyperbolic graphs of bounded valency \cite[Theorem 1.3]{MonodShalom2006}, \cite[Corollaries 7.6 \& 7.8, Theorem 7.13]{MonodShalom2004}, \cite[Theorem 3]{MineyevMonodShalom2004},
	\item  acylindrically hyperbolic groups \cite[Corollary 1.5]{Osin2016}, and
	\item countable groups admitting a non-elementary, metrically proper and essential action on a finite-dimensional irreducible CAT(0) cube complex \cite[Corollary 1.8]{ChatterjiFernosIozzi2016}.
\end{itemize}

\begin{rmks}
	\ 
	\begin{itemize}
		\item Theorem~\ref{thm: normalish} should be compared with the results of \cite[Section~6]{HartmanKalantar2022}, where several classes of groups are shown to admit no probability measure $\mu$ with a non-trivial amenable $\mu$-SRS. Their conclusions are stronger, as they rule out the existence of any non-trivial $\mu$-stationary measure on $\mathrm{Sub}_\mathrm{am}(G)$, including those that may not be a $\mu$-boundary. By contrast, the class of groups covered by Theorem~\ref{thm: normalish} is broader, and the proof is based on a unified argument that applies uniformly to all groups appearing in the statement.
		\item Following \cite[Theorem 5.4]{HartmanKalantar2022}, A. Alpeev showed that a group $G$ is C*-simple if and only if, for a generic probability measure on $G$ (in the Baire category sense), the group $G$ acts essentially freely on the associated Poisson boundary \cite[Theorem B]{Alpeev2025}. Corollary \ref{cor: essentially free} provides explicit hypotheses under which every probability measure has this property, rather than just a Baire-generic subset.
	\end{itemize}
\end{rmks}

Finally, in Section~\ref{sec:normalish} we observe that one cannot expect a converse to Theorem~\ref{thm: normalish}. More precisely, we show that many Baumslag-Solitar groups have amenable normalish subgroups but their space of amenable subgroups is countable.  Thus, they admit no non-trivial amenable $\mu$-SRSs for any non-degenerate $\mu$, see Corollary \ref{cor:noSRSs}. This is the case in particular for $\mathrm{BS}(2,3)=\langle a,t \mid t a^2 t^{-1}=a^3\rangle$. The proof uses an unpublished result of Y. Cornulier \cite{CornulierMSE2021}.

\subsection{Organization}
In Section \ref{section: preliminariesBoundaries} we recall basic properties of boundary actions of groups as well as their connections with C*-simplicity. We also discuss our results and further questions. Section \ref{sec:lampBoundaries} deals with proofs showing the existence of boundary SRSs and in particular with the proof of Theorem \ref{thm: syndetic}. Finally, in Section \ref{sec:normalish} we prove Theorem \ref{thm: normalish} and show that many Baumslag-Solitar groups have amenable normalish subgroups but no amenable $\mu$-SRSs.

\subsection*{Acknowledgments}
We thank Yair Hartman, Mehrdad Kalantar and Nicolás Matte Bon for their feedback during various stages of this project. Anna Cascioli and Eduardo Silva thank Sam Mellick for useful conversations and acknowledge the support of the Deutsche Forschungsgemeinschaft (DFG, German Research Foundation) under Germany's Excellence Strategy EXC 2044/2 –390685587, Mathematics Münster: Dynamics–Geometry–Structure. Martín Gilabert Vio thanks Sasha Bontemps, Damien Gaboriau and Adrien Le Boudec for useful conversations and acknowledges support from the ECOS project 23003 “Small spaces under action”.

\section{Background and discussion}\label{section: preliminariesBoundaries}

We give some background on random walks, stationary boundaries and the space of subgroups of a group. We refer to \cite{Furman2002, LeBoudecMatteBon2018} for more details on this material. We also discuss C*-simple measures and further questions. In this section $G$ is always a countable group.

\subsection*{Random walks and stationary spaces}
Let $\mu$ be a probability measure on $G$. The \emph{Shannon entropy} of $\mu$ is the non-negative quantity $H(\mu) = -\sum_{g \in G} \mu(g) \log \mu(g)$. We denote by $\P_\mu$ the law of the right random walk $\mathbf{w} = (w_n)_{n \geq 0} \in G^\N$ defined by $\mu$, that is, the image of the Bernoulli measure $\mu^\N$ through the map
\begin{align*}
	G^\N & \to G^\N \\
	(g_n)_{n \geq 0}&\mapsto (w_n)_{n \geq 0}
\end{align*}
where $w_n = g_1\cdots g_n$ for each $n\ge 1$. The space $G^\N$ is endowed with a left action of $G$ by multiplication, so that $g(w_n)_{n \geq 0} = (gw_n)_{n \geq 0}$ for all $g \in G$ and $(w_n)_{n \geq 0} \in G^\N$.

Let $X$ be a compact metric space on which $G$ acts by homeomorphisms. A Borel probability measure $\eta$ on $X$ is \emph{$\mu$-stationary} if $\eta = \sum_{g \in G}\mu(g)g_\ast \eta$. The martingale convergence theorem ensures that $\P_{\mu}$-almost surely the limit $\lim_{n \to \infty} (w_n)_\ast \eta = \eta_\mathbf{w}$ exists in $\mathrm{Prob}(X)$. When $\eta_\mathbf{w}$ is $\P_{\mu}$-almost surely a Dirac mass, we call the probability space $(X, \eta)$ a \emph{$\mu$-boundary} of $G$. In this case, we obtain a $G$-equivariant and $S$-invariant map $(G^\N, \P_\mu) \to (X, \eta)$, where $S \colon G^\N \to G^\N$ is the shift map $S((w_n)_{n \geq 0}) = (w_{n+1})_{n \geq 0}$. Conversely, if there is $x \in X$ such that $\P_\mu$-almost surely $\lim_{n \to \infty} w_nx$ exists in $X$, the distribution of $\lim_{n \to \infty} w_nx$ is $\mu$-stationary and defines a $\mu$-boundary on $X$. The literature sometimes refers to $\mu$-boundaries as \emph{$\mu$-proximal systems} \cite{FurstenbergGlasner2010}.

By abuse of notation, a probability space $(X,\eta)$ where $G$ acts by nonsingular measurable transformations is also called a $\mu$-boundary if it is $\mu$-stationary and admits a compact metrizable model which is a $\mu$-boundary in the above sense. The \emph{Poisson boundary} of $(G,\mu)$ is the maximal $\mu$-boundary of $G$, in the sense that any other such space is a $G$-equivariant quotient of it. It is uniquely defined up to $G$-equivariant measurable isomorphisms.

We will use the following well-known result, which is a consequence of the ergodicity of $\mu$-boundaries.
\begin{prop}\label{prop: invariant is delta}
	Let $\mu$ be a probability measure on a countable group $G$ acting by homeomorphisms on a compact metric space $X$. Let $(X,\eta)$ be a $\mu$-boundary of $G$. If either $\eta$ is $G$-invariant or if $\eta$ has atoms, then $\eta$ is a Dirac mass on a point of $X$ which is fixed by the action of $G$.
\end{prop}
\begin{proof}
	Suppose first that $\eta$ is $G$-invariant. Since for $\P_{\mu}$-almost every $\mathbf{w}=(w_n)_{n\ge 0}$ there exists an element $\xi(\mathbf{w}) \in X$ such that $\lim_{n \to \infty}(w_n)_\ast \eta = \delta_{\xi(\mathbf{w})}$, we deduce that $\eta$ is a Dirac mass on a fixed point of the action of $G$ on $X$.
	
	Now suppose that $\eta$ has atoms. By looking at points in $X$ with maximal $\eta$-measure, we deduce that $\eta$ gives positive mass to a finite $G$-orbit $\cO$. The ergodicity of $\eta$ (see \cite[Section 2]{Jaworski1994}, for instance) implies that it gives full mass to $\cO$ and is thus invariant. By the previous case, $\eta$ must be a Dirac mass on a fixed point.
\end{proof}

\subsection*{Hartman-Kalantar's characterizations of C*-simplicity}
Denote the left-regular representation of the group $G$ by $\lambda:G\to \mathcal{U}(\ell^2(G))$. Recall that a \emph{state} on the reduced C*-algebra $C_r^*(G)$ is a positive linear functional $\rho:C_r^*(G)\to \mathbb{C}$ such that $\rho(\lambda_{e_G})=1$. The \emph{canonical trace} $\tau_0$ is the state on $C_r^*(G)$ that satisfies $\tau_0(\lambda_g)=0$ for all $g\in G\setminus \{e_G\}$. The group $G$ acts on a state $\rho$ as $(g\rho)(a)=\rho(\lambda_{g^{-1}ag})$ for each $g\in G$ and, given a probability measure $\mu$ on $G$, the state $\rho$ is called $\mu$-stationary if $\rho=\sum_{g\in G}\mu(g) g\rho$. Since the canonical trace is invariant under the action of $G$, it is in particular $\mu$-stationary for every $\mu\in \mathrm{Prob}(G)$.

Hartman-Kalantar define a \emph{C*-simple measure} as a probability measure $\mu$ on $G$ such that the canonical trace is the unique $\mu$-stationary state on $C_r^*(G)$ \cite[Definition at the top of page 4]{HartmanKalantar2022}. Their first main theorem states that a countable group $G$ is C*-simple if and only if $G$ possesses a C*-simple probability measure \cite[Theorem 5.2]{HartmanKalantar2022}. Their second main theorem shows that a countable group $G$ is C*-simple if and only if $G$ admits a probability measure $\mu$ for which the space $\subam{G}$ of amenable subgroups is uniquely $\mu$-stationary \cite[Corollary 5.7]{HartmanKalantar2022}. The following implication between these two notions of unique stationarity is proved and used in \cite{HartmanKalantar2022}.

\begin{prop}\label{prop: implication c simple}
	Let $\mu$ be a probability measure on a countable group $G$. Suppose that $\subam{G}$ admits a $\mu$-SRS distinct from $\delta_{\{e_G\}}$. Then the space of states on the reduced C*-algebra $C_r^{*}(G)$ is not uniquely $\mu$-stationary.
\end{prop}
\begin{proof}
	Let $\eta$ be an amenable $\mu$-SRS distinct from $\delta_{\{e_G\}}$. By \cite[Lemma 2.3]{HartmanKalantar2022}, the map
	\[
	\rho(\lambda_g)\coloneqq \eta\left( \left\{H\in \subam{G} : g\in H\right\}\right), \ g\in G,
	\]
	extends to a state $\rho$ on $C_r^*(G)$. Furthermore, since $\eta$ is $\mu$-stationary and distinct from $\delta_{\{e_G\}}$, the state $\rho$ is $\mu$-stationary and it differs from the canonical trace.
\end{proof}
Our criterion in Theorem~\ref{thm: syndetic} for the existence of probability measures $\mu\in \mathrm{Prob}(G)$ such that $\subam{G}$ is not uniquely $\mu$-stationary therefore shows that these measures are not C*-simple in the sense of Hartman–Kalantar.
However, we remark that we do not know whether the converse to Proposition \ref{prop: implication c simple} is true. In principle, there may exist probability measures $\mu$ on $G$ that are not C*-simple even though $\subam{G}$ is uniquely $\mu$-stationary.

\subsection{SRSs that are not $\mu$-boundaries}
Let $\mu$ be a probability measure on a countable group~$G$. All results in this paper concern amenable $\mu$-stationary random subgroups for which the space $(\subam{G},\eta)$ is a $\mu$-boundary (see Definition~\ref{def: boundary SRS}).  It follows from work of H.~Furstenberg \cite[Theorem 14.1]{Furstenberg1972} and Glasner-Weiss \cite[Theorem 8.5]{GlasnerWeiss2016} that, if $G$ acts by homeomorphisms on a compact metric space $X$, then the following conditions are equivalent:
\begin{itemize}
	\item For every $\mu$-stationary probability measure $\nu$ on $X$, the space $(X,\nu)$ is a $\mu$-boundary.
	\item There exists a unique $\mu$-stationary probability measure $\nu$ on $X$, and the space $(X,\nu)$ is a $\mu$-boundary.
\end{itemize}

In particular, when $X=\subam{G}$, Question~\ref{question: main question} asks for probability measures $\mu$ on $G$ such that $\subam{G}$ is not uniquely $\mu$-stationary. The equivalence above shows that, in this setting, our methods cannot identify all amenable $\mu$-stationary random subgroups, since there will necessarily exist some that do not arise as $\mu$-boundaries of $G$. This leads to the natural question of whether Theorem~\ref{thm: normalish} extends to arbitrary stationary random subgroups of $G$.

\begin{qnnum}
	Let $\mu$ be non-degenerate probability measure on a countable group $G$, and let $\eta$ be an ergodic $\mu$-stationary probability measure on $\sub{G}$ that is not a point mass on a finite normal subgroup of $G$. 
	\begin{itemize}
		\item Is it true that $\eta$ must be supported on normalish subgroups of $G$?
		\item Suppose further that $\eta$ is supported on amenable subgroups of $G$. Is it true that all $L^2$-Betti numbers of $G$ must vanish?
	\end{itemize}
\end{qnnum}

We remark that if $\eta$ is an IRS of a countable group $G$ supported on subgroups whose $L^2$-Betti numbers all vanish, then it is known that the same conclusion holds for $G$. Indeed, as shown by Abert-Glasner-Virag, every IRS of $G$ arises as the image of a p.m.p.\ action of $G$ under the stabilizer map \cite[Proposition 13]{AbertGlasnerVirag}. This gives rise to the ``principal extension'' of groupoids in the language of \cite[Section 5.3]{SauerThom2010}. It follows from \cite[Lemma 5.1]{SauerThom2010} that the $L^2$-Betti numbers of the stabilizer groupoid vanish. Then \cite[Theorem 1.3]{SauerThom2010} shows that the $L^2$-Betti numbers of the groupoid of the action $G \curvearrowright (\sub{G}, \eta)$ vanish as well, and these coincide with the $L^2$-Betti numbers of $G$  (see \cite[Theorem 5.4]{CarderiGaboriauDeLaSalle2021} for instance). We are grateful to Sam Mellick and Damien Gaboriau for bringing this argument to our attention.

	\section{Existence of amenable boundary $\mu$-stationary random subgroups} \label{sec:lampBoundaries}
	
	This section presents results on the existence of non-trivial amenable boundary stationary random subgroups in countable groups. We begin with a short proof for wreath products (Proposition~\ref{prop: wreath products}), which also serves as motivation for the proof of Theorem~\ref{thm: syndetic}, where we establish a general criterion guaranteeing their existence. We then apply this criterion to Thompson's group $F$ (Corollary~\ref{cor: thompson F}) and to permutational wreath products (Corollary~\ref{cor: permutational wreath products}).

	\subsection{A short proof for wreath products}\label{subsection: wreath}
	Let $A, B$ be countable groups. Their wreath product is \(A \wr B = \bigoplus_B A \rtimes B,\)
	where $\bigoplus_B A$ denotes the group of finitely supported functions
	$\varphi : B \to A$, and the action of $B$ on $\bigoplus_B A$ is given by
	\[
	(x\cdot \varphi)(y) = \varphi(x^{-1}y), \qquad x,y\in B .
	\]
	Thus the multiplication in $A \wr B$ is
	\(
	(\varphi,b)(\psi,c)
	=
	\bigl(\varphi \cdot (b\cdot \psi),\, bc\bigr),
	\) where $(b\cdot\psi)(y)=\psi(b^{-1}y)$, for each $b,c\in B$ and $\varphi,\psi\in \bigoplus_B A$.
	
	Given a probability measure $\mu \in \mathrm{Prob}(A \wr B)$, the $\mu$-random walk $\mathbf{w}=(w_n)_{n\ge0}$ on $A \wr B$ can be written using the semidirect product structure as \(w_n=(\varphi_n,b_n).\) Here $\varphi_n\in \bigoplus_B A$ is called the \emph{lamp configuration at
		time $n$}, and $b_n\in B$ is the \emph{position at time $n$}. In general $(\varphi_n)_{n\ge 0}$ does not define a random walk on $\bigoplus_B A$, whereas $(b_n)_{n\ge 0}$ is a random walk on the base group $B$.
	
	\begin{defn}
		We say that the lamp configurations of the $\mu$-random walk on $A\wr B$ \emph{stabilize almost surely} if there exists a conull set of trajectories $\mathbf{w}=(\varphi_n,b_n)_{n\ge0}$ such that for every
		$b\in B$ there exists $N$ for which \( \varphi_n(b)=\varphi_N(b)\) for all \( n\ge N .\)
	\end{defn}
	
	Under this assumption, there is an almost everywhere defined map \( (A\wr B)^{\mathbb N} \to A^B \)
	which assigns to $\mathbb P_\mu$-almost every trajectory $\mathbf w=(\varphi_n,b_n)_{n\ge0}$ the \emph{limit lamp configuration} $\varphi_\infty(\mathbf w)$ defined by
	\[
	\varphi_\infty(\mathbf w)(b)=\lim_{n\to\infty}\varphi_n(b), \qquad b\in B.
	\]
	
	Lamp configurations do not always stabilize: for every $\varepsilon>0$
	there exist probability measures $\mu$ with finite $(1-\varepsilon)$-moment
	for which stabilization fails \cite{Kaimanovich1983}. However, a sufficient
	condition for stabilization is that $\mu$ has finite first moment and the
	induced random walk on $B$ is transient \cite{Erschler2011}. If $B$ is
	finitely generated with at least cubic growth, then a theorem of
	Varopoulos \cite{Varopoulos1986} implies that every non-degenerate
	probability measure on $A\wr B$ induces a transient random walk on $B$.
	In particular, in this case one may choose $\mu$ to be finitely supported.
	
	\begin{prop}\label{prop: wreath products}
		Let $A,B$ be non-trivial countable groups and let $\mu$ be a non-degenerate
		probability measure on $A\wr B$ such that the lamp configurations of the
		$\mu$-random walk stabilize almost surely. Then there exists a measure
		$\eta \neq \delta_{\{e\}}$ on $\mathrm{Sub}_{\mathrm{am}}(A\wr B)$ such that
		$(\mathrm{Sub}_{\mathrm{am}}(A\wr B),\eta)$ is a $\mu$-boundary.
		
		If in addition $A$ is non-abelian, then $\eta$ is non-atomic.
	\end{prop}
	
	\begin{proof}
		Fix $a\in A\setminus\{e_A\}$. For $b\in B$ and $x\in A$ denote by
		$\delta_b^x\in\bigoplus_B A$ the configuration
		\[
		\delta_b^x(y)=
		\begin{cases}
			x & y=b,\\
			e_A & y\neq b 
		\end{cases}
		\]
		that takes the value $x$ at $b$ and the identity $e_A$ elsewhere. Let $\mathbb P_\mu$ be the law of the $\mu$-random walk.
		Given a trajectory
		$\mathbf w=(\varphi_n,b_n)_{n\ge0}$ with limit lamp configuration
		$\varphi_\infty(\mathbf w)$, define
		\[
		H(\mathbf w)
		=
		\Big\langle
		\delta_b^{\,\varphi_\infty(\mathbf w)(b)a(\varphi_\infty(\mathbf w)(b))^{-1}}
		:\; b\in B
		\Big\rangle
		\subseteq \bigoplus_B A .
		\]
		That is, the subgroup $H(\mathbf{w})$ is obtained by placing at each position $b$
		the element $a$ conjugated by the limiting lamp value $\varphi_{\infty}(\mathbf{w})(b)$. The map $\mathbf w\mapsto\varphi_\infty(\mathbf w)$ is
		$A\wr B$-equivariant. 
		
		Let $g=(\varphi,b')\in A\wr B$. Then
		\[
		(g\cdot \varphi_\infty)(y)
		=
		\varphi(y)\,\varphi_\infty(b'^{-1}y), \text{ for each }y\in B.
		\]
		Therefore
		\[
		H(g\mathbf w)
		=
		\Big\langle
		\delta_y^{\,\varphi(y)\varphi_\infty(b'^{-1}y)a(\varphi(y)\varphi_\infty(b'^{-1}y))^{-1}}
		:\; y\in B
		\Big\rangle,
		\]
		and writing $y=b'b$ yields
		\[
		H(g\mathbf w)
		=
		\Big\langle
		\delta_{b'b}^{\,\varphi(b'b)\varphi_\infty(b)a(\varphi(b'b)\varphi_\infty(b))^{-1}}
		:\; b\in B
		\Big\rangle
		=
		g\,H(\mathbf w).
		\]
		Thus the map $\mathbf w\mapsto H(\mathbf w)$ is $A\wr B$-equivariant.
		The pushforward \( \eta = (H)_*\mathbb P_\mu\) is therefore a $\mu$-boundary supported on
		$\mathrm{Sub}_{\mathrm{am}}(A\wr B)$. Since $a\neq e_A$, the subgroup
		$H(\mathbf w)$ is almost surely non-trivial, so
		$\eta\neq \delta_{\{e\}}$.
		
		If $A$ is non-abelian we may choose $a$ outside the center of $A$. In this case limit configurations with lamps in different cosets of the centralizer of $a$ in $A$ yield different subgroups $H(\mathbf w)$, and hence $\eta$ is non-atomic.
	\end{proof}

	The group $A \wr B$ is known to be C*-simple whenever $A$ is C*-simple; see, for example, the short proof in \cite[Proposition 5.5]{BedosOmland2018}. Alternatively, using \cite[Theorem 1.5]{KalantarKennedy2017}, one can prove the C*-simplicity of $A \wr B$ by exhibiting a topologically free $(A \wr B)$-boundary. That is, it suffices to show the existence of a compact space $X$ equipped with a minimal action of $A \wr B$ by homeomorphisms such that the closure of every $(A \wr B)$-orbit in $\mathrm{Prob}(X)$ contains a Dirac delta measure, and such that the set of fixed points in $X$ of every non-trivial element of $A \wr B$ has an empty interior. In order to do so, assume that $A$ is C*-simple and let $X_A$ be a topologically free $A$-boundary. Consider the action of $A \wr B$ on $\prod_B X$ given by $(\varphi,c)\cdot (x_b)_{b \in B} = (\varphi_b\cdot x_{c^{-1}b})_{b \in B}$ for every $\varphi\in \bigoplus_B A$, $c\in B$ and $(x_b)_{b \in B}\in \prod_B X_A$. This yields a topologically free boundary action of $A \wr B$, and hence $A \wr B$ is C*-simple.
	
	As a consequence of Proposition~\ref{prop: wreath products}, we obtain the following. 
	
	\begin{cor}\label{cor: csimple infinite first moment}
		Let $A$ be a countable C*-simple group and let $B$ be a finitely generated group of at least cubic growth. Then every C*-simple probability measure on $A \wr B$, in the sense of \cite{HartmanKalantar2022}, has infinite first moment.
	\end{cor}

	\subsection{Record times and the proof of Theorem \ref{thm: syndetic}}\label{subsection: proof syndetic}
	
	For the construction of the measure appearing in Theorem \ref{thm: syndetic} we first need to introduce some general terminology.
	
	\begin{defn}\label{def: ev simple records}
		Let $p = (p_j)_{j\geq 0}$ be a probability measure on $\N$, and let $(X_n)_{n\geq 1}$ be a sequence of i.i.d.\ random variables with law $p$. For every $n \in \N_{>0}$ let us denote by $M_n \coloneqq \max\{X_1,X_2,\ldots, X_n\}$ the \emph{record value of the sequence at time $n$}. We say that the record value $M_n$ is \emph{simple} if
		\[
		\abs{ \{ i\in \{1,\ldots, n\} : X_i=M_n \} } = 1.
		\]
		We say that the records of $(X_n)_{n \geq 1}$ are \emph{eventually simple} if $p^{\otimes \N}$-almost surely there is $N \geq 1$ such that, for all $n \geq N$, the record value $M_n$ is simple. Finally, define the sequence $(T_k)_{k \geq 1}$ of \emph{record times of $(X_n)_{n \geq 1}$} by $T_1 = 1$ and $T_{k+1} = \min \{n > T_k : X_n = M_n\}$ for $k \in \N_{> 0}$.
	\end{defn}
	The following lemma can be found in \cite[Lemma 2.3 \& Corollary 2.6]{ErschlerKaimanovich2023}. The sufficient condition is due to \cite[Theorem 3.2]{BrandsSteutelWilms1994}, while the necessary condition is due to \cite[Theorem 2]{Qi1997}. See also the proof given in \cite[Theorem 3, Corollaries 3.1 \& 3.2]{Eisenberg2009}.
	\begin{lema}\label{lem: ev simple records}
		Let $p$ be a probability measure on $\N$ and let $(X_n)_{n\geq 1}$ be a sequence of i.i.d.\ random variables with law $p$. Then $(X_n)_{n\ge 1}$ has eventually simple records if and only if $p$ has an infinite support and
		\begin{equation}\label{eq: finite sum simple records}
			\sum_{j=0}^{\infty}\left(\frac{p_j}{p_j+p_{j+1}+\cdots}\right)^2<\infty.
		\end{equation}
		The latter condition holds in particular when $p_j$ has polynomial decay as $j\to \infty$.
	\end{lema}
	
	\begin{lema}[{\cite[Lemma 2.17]{ErschlerKaimanovich2023}}]\label{lem: gauge}
		For any probability measure $p$ on $\N$ with infinite support there exists a non-decreasing function $\Phi:\N\to \N$ such that almost surely 
		$T_{k+1}\le \Phi(R_k)$ holds for all sufficiently large $k$.
	\end{lema}
	
	We will prove the following reformulation of Theorem \ref{thm: syndetic}, which makes the construction of the measure and the resulting boundary SRS explicit. Recall that, for a given probability measure $\mu$ on $G$, we denote by $\P_\mu$ the law of the $\mu$-random walk $(w_n)_{n\ge 0}$ on $G$.
	\begin{teo}
		Let $G$ be a countable group, and suppose that $G$ has a non-trivial subgroup $H$ such that for all finite subsets $Q, Z \subseteq G$ there is $b \in G$ such that
		\[
		bZ,b^{-1}Z\subseteq \left \{g \in G : Q \cap H = Q \cap gHg^{-1}\right \}.
		\]
		Then there exists a non-degenerate and symmetric probability measure $\mu$ on $G$ with finite Shannon entropy such that $\P_\mu$-almost surely the sequence $(w_nHw_n^{-1})_{n \geq 0}$ converges to a non-trivial subgroup of $G$. In particular, if $H$ is amenable, the distribution of the subgroup $\lim_{n\to \infty}w_n H w_n^{-1}$ defines an amenable boundary $\mu$-SRS of $G$.
	\end{teo} 
	\begin{proof}
		Consider a probability measure $p$ on $\N$ that gives positive mass to all non-negative integers and that satisfies Condition \eqref{eq: finite sum simple records}, so that Lemma \ref{lem: ev simple records} ensures that a sequence of i.i.d.\ random variables with law $p$ almost surely has eventually simple records. Moreover, we may choose $p$ such that its Shannon entropy $H(p)=-\sum_{j\ge 0}p_j\log(p_j)$ is finite.
		
		Let us choose a sequence $(\widetilde{A}_n)_{n \geq 0}$ of pairwise disjoint and symmetric subsets of $G$ such that $\widetilde{A}_0=\{e_G\}$, $|\widetilde{A}_n|\le 2$ for all $n\ge 0$ and $G=\bigcup_{n\ge 0} \widetilde{A}_n$.
		We start by constructing inductively increasing sequences of finite subsets $(A_n)_{n\ge 0}$, $(\triangle_n)_{n\ge 0}$, $(Q_n)_{n\ge 0}$ of $G$, along with a sequence of elements $b_n \in G$ for every $n \ge 0$.
		
		We first set $A_0=\widetilde{A}_0=\{e_G\}$, $\triangle_0=Q_0=\varnothing$ and $b_0=e_G$. Now, let $n\ge 1$ and suppose that the sets $\triangle_i$ and $Q_i$ together with the element $b_i$ have already been defined for every $0\le i\le n$. We set
		\begin{equation*}
			A_{n+1}= \widetilde{A}_{n+1}\setminus\bigcup_{i=0}^n \{b_i,b_i^{-1}\}, \
			\triangle_{n+1}=\left(\bigcup_{i=0}^n A_i\cup \left\{b_i,b_i^{-1}\right\}\right)^{\Phi(n)},
		\end{equation*}
		where $\Phi$ is the gauge function associated with $p$ from Lemma~\ref{lem: gauge}.
		Next, we set $Q_{n+1}$ to be a finite subset of $G$ containing $Q_n$ and such that for every $d\in \triangle_n$ we have $Q_{n+1}\cap dHd^{-1}\neq\{e_G\}$. Applying the hypothesis of the theorem to the finite subsets $Q\coloneqq\bigcup_{d\in \triangle_n} d^{-1}Q_{n+1}d$ and $Z\coloneqq\triangle_n$, we find $b_{n+1}\in G$ such that for all $z \in \triangle_n$ we have
		\[
		Q \cap H= Q \cap (b_{n+1}z)H (b_{n+1}z)^{-1}= Q\cap (b_{n+1}^{-1}z)H (b_{n+1}^{-1}z)^{-1}.
		\]
		As a consequence, notice that for each $d,z\in \triangle_n$ we have
		\begin{equation}\label{eq: Qconsequence}
			Q_{n+1}\cap dHd^{-1}=Q_{n+1}\cap (db_{n+1}z)H(db_{n+1}z)^{-1}=Q_{n+1}\cap (db^{-1}_{n+1}z)H(db^{-1}_{n+1}z)^{-1}.
		\end{equation}
		This completes the inductive construction of the sequences $(A_n)_{n\ge 0}$, $(\triangle_n)_{n\ge 0}$, $(Q_n)_{n\ge 0}$ and $(b_n)_{n\ge 0}$.

		Let us now fix an auxiliary sequence $(\alpha_i)_{i\ge 1}$ with $0<\alpha_i<1$ for every $i\ge 1$ and such that $\sum_{i\ge 1}\alpha_i<\infty.$ Define a symmetric probability measure $\mu$ on $G$ with $\mu(A_0)=p_0$ and such that for every $i\ge 1$ we have
		\begin{equation}\label{eq: def mu}
			\mu\left(A_i\cup \{b_i,b^{-1}_{i}\}\right)=p_i\text{ and }
			\mu\left(A_i\backslash\{b_i,b_i^{-1}\}\right)\le \alpha_i p_i.\end{equation}

		In what remains of the proof, we will justify that such a measure $\mu$ satisfies the conclusions of the theorem.
		
		First, notice that $\mathrm{supp}(\mu)=G$ because $p$ is fully supported and $G=\cup_{n\ge 0} A_n$. In particular, the probability measure $\mu$ is non-degenerate. Moreover, $\mu$ has finite Shannon entropy since
		\begin{equation}\label{eq: ineq entropy}
			H(\mu)\le \log(4)+H(p)<\infty.
		\end{equation}
		Indeed, to see this it suffices to notice that $\mu$ can be sampled in two steps: first, choose an index $i\in \N$ according to the distribution $p$, then sample $g_i \in A_i\cup \{b_i,b_i^{-1}\}$ according the conditional distribution of $\mu$ on that subset. Inequality \eqref{eq: ineq entropy} then follows from the standard (conditional) entropy decomposition $H(X)=H(X\mid Y)+H(Y)$ for random variables $X$ and $Y$, together with the fact that the entropy of a random variable taking values in a set of size $N$ is at most $\log(N)$. Since $|A_i \cup \{b_i, b_i^{-1}\}| \le 4$ for every $i \ge 0$, we obtain the desired bound.

		It now remains to show the $\P_\mu$-almost sure convergence of the sequence $(w_n H w_n^{-1})_{n\ge 0}$ to a non-trivial subgroup of $G$. 
		Let $(g_n)_{n\ge 1}$ be an independent sequence of elements of $G$ distributed according to $\mu$ and denote $w_n=g_1\cdots g_n$ for every $n\ge 1$. In other words, the process $(w_n)_{n\ge 1}$ has the law $\P_\mu$ of the $\mu$-random walk on $G$. For every $i\ge 1$, denote $X_i\in \N$ the unique value such that $g_i\in A_{X_i}\cup \{b_{X_i},b_{X_i}^{-1}\}$. Then, it follows from the construction of $\mu$ that the sequence $(X_i)_{i\ge 1}$ is i.i.d.\ with law $p$, and hence has eventually simple records. Denote by $(T_k)_{k\ge 1}$ the associated record times and by $R_k\coloneqq X_{T_k}$ the record values for every $k\ge 1$.
		By combining
		\begin{enumerate}[label=(\alph*)]
			\item that $\mu(A_i\backslash\{b_i,b_i^{-1}\})\le \alpha_i p_i$ (Condition \eqref{eq: def mu}) with $\sum_{i\ge 1}\alpha_i<\infty$ and the Borel-Cantelli lemma,
			\item Lemma \ref{lem: gauge}, and
			\item the definition of eventually simple record times (Definition \ref{def: ev simple records}),
		\end{enumerate}
		we conclude that almost surely there exists $k_0\ge 1$ such that for all $k\ge k_0$ we have
		\begin{enumerate}[label=(\Alph*)]
			\item\label{item: cond1} $g_{T_k}\in \{b_{R_k},b_{R_k}^{-1}\}$,
			\item\label{item: cond2} $T_{k+1}\le \Phi(R_k)$, and
			\item\label{item cond3} $X_m<R_k$ for all $T_k<m<T_{k+1}$.
		\end{enumerate}
		
		Now, let $n>T_{k_0}$ and choose $k\ge k_0$ such that $T_k\le n<T_{k+1}$. Then, we can write
		\begin{equation*}
			w_n=g_1\cdots g_{T_{k}-1}\cdot g_{T_k} \cdot g_{T_k+1}\cdots g_n.
		\end{equation*}
		Set $d=g_1\cdots g_{T_{k}-1}$ and $z=g_{T_k+1}\cdots g_n$, so that $w_n=dg_{T_k}z$. It follows from Conditions \ref{item: cond2} and \ref{item cond3} above that
		\begin{equation*}
			d,z\in \left( \bigcup_{i=0}^{R_k-1}A_i\cup \{b_i,b_i^{-1} \}\right)^{T_{k+1}}\subseteq \triangle_{R_k}.
		\end{equation*}
		In addition, thanks to Condition \ref{item: cond1} together with the inductive construction of $b_{R_k}$, we also have that
		\begin{equation*}
			Q_{R_k}\cap dHd^{-1}=Q_{R_k}w_nHw_{n}^{-1},
		\end{equation*}
		as in Equation \eqref{eq: Qconsequence}.
		
		\begin{claim}
			For every $m\ge n$, we have 
			\[
			Q_{R_k}\cap dHd^{-1}=Q_{R_k}\cap w_mH w_m^{-1}.
			\]
		\end{claim}
		\begin{proof}[Proof of the claim]
			Let $\ell$ be such that $T_\ell\le m < T_{\ell +1}$. Notice that since $m>n$, we must have $\ell\ge k$. We prove the claim by induction on $\ell$.
			
			If $\ell=k$, then the claim was already shown to hold in Equation \eqref{eq: Qconsequence}. Now suppose that $\ell>k$ and write $w_m=\widetilde{d}g_{T_\ell}\widetilde{z}$, where
			\[
			\widetilde{d}=g_1\cdots g_{T_\ell-1} \text{ and }\widetilde{z}=g_{T_{\ell}+1}\cdots g_m.
			\]
			As in the argument preceding the claim, we have 
			\[
			Q_{R_\ell}\cap \widetilde{d}H\widetilde{d}^{-1}=Q_{R_\ell}\cap w_m H w_m^{-1},
			\]
			thus in particular 
			\begin{equation}\label{eq: Q almost finish}
				Q_{R_k}\cap \widetilde{d}H\widetilde{d}^{-1}=Q_{R_k}\cap w_m H w_m^{-1}.
			\end{equation}
			Applying the inductive hypothesis with $\widetilde{d}=w_{T_\ell-1}$ gives
			\begin{equation}\label{eq: Q finish}
				Q_{R_k}\cap dHd^{-1}=Q_{R_k}\cap \widetilde{d}H\widetilde{d}^{-1}.
			\end{equation}
			Combining Equations \eqref{eq: Q almost finish} and \eqref{eq: Q finish} concludes the proof of the claim.
		\end{proof}
		The compactness of $\mathrm{Sub}(G)$ implies that the sequence $(w_n H w_n^{-1})_{n \ge 0}$ has at least one limit point $K \in \mathrm{Sub}(G)$. Since the sets $Q_{R_k}$ form an exhaustion of $G$, the previous claim shows that, for every open neighborhood $U$ of $K$, the sequence $(w_n H w_n^{-1})_{n \ge 0}$ eventually remains in $U$. Hence the limit point is unique, and $(w_n H w_n^{-1})_{n \ge 0}$ converges in $\mathrm{Sub}(G)$. Finally, the limit subgroup $K$ is non-trivial: indeed, $Q_{R_k} \cap d H d^{-1}$ contains a non-trivial element $g \in G \backslash \{e_G\}$, and the previous claim shows that $g \in K$.
	\end{proof}
	\begin{rmk}\label{remark: dynamical interpretation}
		Theorem~\ref{thm: syndetic} requires the existence of a non-trivial subgroup $H \subseteq G$ such that for every pair of finite subsets $Q, Z \subseteq G$ there exists $b \in G$ with
		\[
		bZ,\, b^{-1}Z \subseteq \{ g \in G : Q \cap H = Q \cap gHg^{-1} \}.
		\]
		
		This condition admits a dynamical interpretation for the conjugation action of $G$ on $\mathrm{Sub}(G)$. The finite set $Q$ determines a basic neighbourhood of $H$ in the Chabauty topology,
		\[
		U_Q(H) = \{K \in \sub{G} : Q \cap K = Q \cap H\}.
		\]
		The above property then asserts that, for every such neighbourhood and every finite set $Z \subseteq G$, there exists $b \in G$ such that
		\[
		bz \cdot H,\, b^{-1}z \cdot H \in U_Q(H)
		\qquad \text{for all } z \in Z .
		\]
		In other words, any finite portion of the $G$-orbit of $H$ in $\mathrm{Sub}(G)$ can be simultaneously pushed arbitrarily close to $H$ by conjugation with a single group element $b$ as well as by its inverse. This can be viewed as a strong form of recurrence for the action of $G$ on the orbit of~$H$.
	\end{rmk}

	\subsection{Amenable boundary $\mu$-SRSs for Thompson's group $F$}
	
	Consider the group of orientation-preserving homeomorphisms $f \colon \R\to \R$ such that there exists a finite subset $D \subset \Z\left[\frac{1}{2}\right]$ with the property that for each bounded connected component $C$ of $\R \setminus D$, the map $f$ restricted to $C$ is of the form
	\begin{itemize}
		\item $x \mapsto 2^kx + q$ where $k \in \Z$ and $q \in \Z[1/2]$ if $C$ is bounded, and of the form
		\item $x \mapsto x + m$ where $m \in \Z$ if $C$ is unbounded.
	\end{itemize}
	It is well known that the above describes a group isomorphic to Thompson's group $F$, the group of all increasing piecewise dyadically affine homeomorphisms of the closed unit interval $[0,1]$; see, e.g., \cite[Section 3.1]{BelkBrown2005}. In this representation, the commutator subgroup $[F,F]$ coincides with the homeomorphisms in $F$ which are the identity outside of a compact interval of $\R$.
	
	\begin{proof}[Proof of Corollary \ref{cor: thompson F}]
		We show that Thompson's group $F$ satisfies the hypotheses of Theorem \ref{thm: syndetic}. Let $f \colon \R\to \R$ be an arbitrary non-trivial element of $F$ whose support is contained in a proper closed subinterval of $[0,1]$. For each $k\in \Z$ denote by $t_k\in F$ the translation by $k \in \Z$. Let us consider the subgroup $H$ of $F$ generated by all elements $t_k f t_k^{-1}$, $k\in \Z$. These elements commute since they have disjoint supports, so $H$ is abelian. Notice also that $H$ and all of its conjugates are contained in the normal subgroup $[F,F]$. Hence, to verify the condition of Theorem \ref{thm: syndetic}, it suffices to show that for any finite subset $Q\subseteq [F,F]$ and every finite subset $Z\subseteq F$ we can find $b\in F$ such that
		\[
		bZ, b^{-1}Z \subseteq \{g\in F : Q\cap H= Q\cap gHg^{-1}\}.
		\]
		
		Since $Q \subseteq [F,F]$, all elements of $Q$ are the identity outside of a sufficiently large interval. Thus we can choose $N\in \N$ such that $\mathrm{supp}(q)\subseteq [-N,N]$ for every $q\in Q$. Next, choose $M\ge 1$ large enough so that for every $z\in Z$ there exist $m_{z}^{-},m_{z}^{+}\in \Z$ such that
		\[
		z(x)=x+m_{z}^{-} \text{ for } x\le -M,  \ z(x)=x+m_{z}^{+} \text{ for }  x\ge M.
		\]
		For every $z\in Z$, denote by $r_z^{-},r_z^{+} \in F$ the maps $r_z^{-}(x)=x-m_z^{-}$ and $r_z^{+}(x)=x-m_z^{+}$ for $x\in \R$. Then $\mathrm{supp}(zr_z^{-})\subseteq [-M,+\infty)$ and $\mathrm{supp}(zr_z^{+})\subseteq (-\infty,M]$ for each $z\in Z$.
		
		Now let $b\in F$ be the translation by $-M-N-1$. Using the fact that $H$ is normalized by any translation, for every $z \in Z$ we have
		\begin{align*}
			bzHz^{-1}b^{-1}&=b(zr^{+}_z)H((r_z^{+})^{-1}z^{-1})b^{-1}  \\
			&=(b(zr^{+}_z)b^{-1}) H (b(zr^{+}_z)b^{-1})^{-1}.
		\end{align*}
		
		The support of $b(zr^{+}_z)b^{-1}$ is contained in $(-\infty, -N-1]$, while all elements in $Q$ have support contained in $[-N,N]$. Therefore, $b(zr^{+}_z)b^{-1}$ commutes with all elements in $Q$, implying that $bZ$ is contained in $\{g\in F : Q\cap H=Q\cap gHg^{-1}\}$. An analogous argument applies to $b^{-1}Z$. This verifies the hypotheses of Theorem \ref{thm: syndetic}, and the conclusion follows from the fact that the set of abelian subgroups of $\mathrm{Sub}(F)$ is closed in the Chabauty topology.
	\end{proof}

	\subsection{Amenable boundary $\mu$-SRSs for permutational wreath products}
	Let $A, B$ be countable groups and let $B \curvearrowright X$ be an action on a countable set $X$. The \emph{permutational wreath product} $G = A \wr_X B$ is the semidirect product $\left(\bigoplus_{x \in X} A\right) \rtimes B$, where $B$ acts on $\bigoplus_{x \in X} A$ by $b.(\varphi_x)_{x \in X} = (\varphi_{b^{-1}.x}))_{x \in X}$ for $b \in B$, $(\varphi_x)_{x \in X} \in \bigoplus_{x \in X}B$. Here, the group $B$ is identified with its canonical copy inside $A\wr_X B$, whereas $A$ is identified with its copy within the direct sum at a given fixed basepoint $o\in B$.
	
	Let $\mu \in \mathrm{Prob}(A \wr_X B)$. As in the case of wreath products, we say that the lamp configurations of the $\mu$-random walk on $A \wr_X B$ stabilize almost surely if there is an infinite orbit $\cO \subseteq X$ such that for $\P_{\mu}$-almost every trajectory $\mathbf{w} = (\varphi_n, b_n)_{n\geq 0}$ the configurations $(\varphi_n)_{n \geq 0}$ restricted to $\cO$ converge pointwise to a limit function $\varphi_\infty(\mathbf{w}) \in A^\cO$. 
	Stabilization of lamps occurs whenever the induced random walk $(b_n.x)_{n \geq 0}$ on $\cO$ is transient for every $x \in \cO$ and $\mu$ has finite first moment. In general, transience of $(b_n.x)_{n \geq 0}$ on $\cO$ can be more delicate than transience of the random walk on $(b_n)_{n \geq 0}$ on $B$ itself. The following proposition applies the criterion in Theorem \ref{thm: syndetic} to construct a symmetric measure $\mu$ with finite entropy such that this is the case.
	
	\begin{cor}\label{cor: permutational wreath products}
		Let $A, B$ be countable groups and let $B \curvearrowright X$ be a faithful action on a countable set $X$ with at least one infinite orbit $\cO$. Suppose that $A$ is not abelian. Then there exists a non-degenerate and symmetric probability measure on $A\wr_X B$ with finite Shannon entropy that admits an amenable boundary $\mu$-SRS distinct from $\delta_{\{e\}}$.
	\end{cor}
	
	\begin{proof}
		
		Let us fix an element $a \in A$ that is not in the center of $A$, and consider the non-trivial amenable subgroup $H = \bigoplus_{x \in \cO} \langle a \rangle$ of $A\wr_X B$. We will show that $H$ satisfies the hypotheses of Theorem \ref{thm: syndetic}.

		First, we note that $H$ is normalized by both subgroups $\bigoplus_{x \in X \setminus \cO} A$ and $B$. Furthermore, $H$ is contained in the normal subgroup $\bigoplus_{x \in \cO} A$. In order to verify the conditions of Theorem \ref{thm: syndetic}, it suffices to show that for any finite subsets $Z,Q \subset \bigoplus_{x \in \cO} A$ we can find $b\in B$ such that  
		\[
		bZ,b^{-1}Z\subseteq \left\{g\in A\wr_X B : Q\cap H=Q\cap gHg^{-1}\right\}.
		\]
		
		Let $Z,Q \subset \bigoplus_{x \in \cO} A$ be finite subsets. Denote by $F_Z, F_Q \subset \cO$ the union of the supports of elements in $Z, Q$, respectively.
		\begin{claim}
			There exists $b \in B \subseteq A \wr_X B$ such that both $bF_Z$ and  $b^{-1}F_Z$ are disjoint from $F_Q$.
		\end{claim}
		\begin{proof}[Proof of the claim]
			Let $\mu$ be a non-degenerate symmetric probability measure on $B$, and denote by $\P_\mu$ the law of the $\mu$-random walk $(w_n)_{n \geq 0}$ on $B$. Since $\cO$ is infinite, it follows from \cite[Lemma 3.1]{ChawlaFrisch2025} (see also \cite[Claim 2.3]{GorokhovskyMatteBonTamuz2024}, or \cite[Lemma 3.2]{LeBoudecMatteBon2018}) that
			\[
			\lim_{n \to \infty} \P_\mu[w_n\cdot x = y] = 0
			\]
			for every $x \in F_Z$ and $y \in F_Q$. A union bound then gives 
			\[
			\lim_{n \to \infty} \P_\mu[w_n\cdot F_Z \cap F_Q = \emptyset] = 1.
			\]
			For every $n \ge 1$, the law of the $n$-th step of the random walk $w_n = g_1g_2\cdots g_n$ is $\mu^{\ast n}$. Since $\mu$ is symmetric, this is the same distribution as $w_n^{-1} = g_n^{-1}g_{n-1}^{-1}\cdots g_1^{-1}$. It then follows that
			\[
			\lim_{n \to \infty} \P_\mu[(w_n)^{-1}\cdot F_Z \cap F_Q = \emptyset] = 1.
			\]
			Another application of a union bound shows the existence of $b \in B$ such that
			\[
			b.F_Z \cap F_Q = b^{-1}\cdot F_Z \cap F_Q = \emptyset. \qedhere
			\]
		\end{proof}
		
		By an argument analogous to the one in the proof of Corollary \ref{cor: thompson F}, for every $z \in Z$ we can write
		\[
		bzHz^{-1}b^{-1} = (bzb^{-1}) Hb(bzb^{-1})^{-1},\]
		and similarly
		\[
		b^{-1}zHz^{-1}b = (b^{-1}zb) H (b^{-1}zb)^{-1}.
		\]
		By the previous claim, $bzb^{-1}$ and $b^{-1}zb$ both commute with every element in $Q$, so that $bZ$ and $b^{-1}Z$ are contained in $\{g \in A\wr_X B : Q \cap H = Q \cap gHg^{-1}\}$. We conclude that there exists an amenable boundary $\mu$-SRS distinct from $\delta_{\{e\}}$.
	\end{proof}

	Permutational wreath products share some algebraic similarities with the usual wreath products discussed in Subsection~\ref{subsection: wreath}, but their geometric and dynamical properties can differ significantly, particularly regarding the behavior of random walks.
	
	Consider a random walk on $A \wr_X B$, and let $b_n = x_1 \cdots x_n$ be its projection to $B$ for every $n \ge 1$, where $(x_n)_{n\ge1}$ is the sequence of independent increments in $B$. In the case of usual wreath products, the walk can be visualized as a lamplighter moving along the Cayley graph of the base group $B$, modifying lamps near the vertices it visits. For permutational wreath products, however, such an interpretation in terms of a walk on the Schreier graph of $X$ does not work. Indeed, the lamp configurations are modified along the so-called \emph{inverted orbit}
	\[
	o, \quad x_1^{-1}.o, \quad x_1^{-1}x_2^{-1}.o, \quad x_1^{-1}x_2^{-1}x_3^{-1}.o, \ \ldots,
	\]
	which does not satisfy the Markov property and does not follow a connected path in the Schreier graph. This makes the analysis of random walks more subtle, but also gives rise to richer geometric phenomena. For instance, when $B=G_{012}$ is the first Grigorchuk group and $X$ is a particular orbit of its action on the boundary of the rooted binary tree, the growth of inverted orbits was used by Bartholdi-Erschler to show that certain permutational wreath products, including $\mathbb{Z}/2\mathbb{Z}\wr_X G_{012}$, have intermediate growth, and to compute the asymptotics of their growth functions \cite[Theorem A]{BartholdiErschler2012}. These were the first groups of intermediate growth for which such asymptotics were determined. Later, the same authors used permutational wreath products to construct examples of finitely generated groups of exponential growth, such as $\mathbb{Z}/2\mathbb{Z}\wr_{X\times X} (G_{012}\times G_{012})$, for which every finitely supported probability measure (not necessarily non-degenerate) has a trivial Poisson boundary \cite[Theorem A]{BartholdiErschler2017}. 
	This answered a previously open question by showing that not every group of exponential growth admits a finitely supported measure with non-trivial Poisson boundary.

\section{Boundary SRSs are supported on normalish subgroups} \label{sec:normalish}
In this section we prove Theorem \ref{thm: normalish}, which relates normalish subgroups of a group $G$ to boundary SRSs on $\mathrm{Sub}(G)$. We also show that many C*-simple Baumslag-Solitar groups have amenable normalish subgroups even though their space of amenable subgroups is countable, so all their amenable $\mu$-SRSs are trivial. We finally show a different proof of triviality of amenable $\mu$-boundaries for groups with property (CS), an operator-algebraic property of groups formulated in terms of their unitary representations which was introduced in \cite{BreuillardKalantarKennedyOzawa2017}.

\subsection{The proof of Theorem \ref{thm: normalish}}
Recall that the \emph{FC-center} of a group $G$ is the subgroup $\mathrm{FC}(G)$ consisting of all elements of $G$ with finite conjugacy class. A group action $G \curvearrowright X$ by homeomorphisms of a compact Hausdorff space $X$ is called \emph{equicontinuous} if the image of $G$ in $\mathrm{Homeo}(X)$ is relatively compact for the compact-open topology.

\begin{proof}[Proof of Theorem \ref{thm: normalish}]
	Let $\mu$ be a non-degenerate probability measure on $G$ and let $\eta$ be a $\mu$-stationary probability measure on $\mathrm{Sub}(G)$ such that the space $(\sub{G},\eta)$ is a $\mu$-boundary of $G$. If $\eta$ is atomic, then it must be a Dirac mass on an infinite normal subgroup, and the result follows. We therefore assume that $\eta$ is non-atomic. Let $g_1, \ldots, g_m \in G$. We will show that for every $\varepsilon>0$ we have
	\begin{equation}\label{eq: normalish eps}
		\eta\left(\left\{ H \in \mathrm{Sub}(G) : \bigcap_{i=1}^m g_i Hg_i^{-1}\text{ is finite} \right\} \right)\le \varepsilon.
	\end{equation}
	Since $\varepsilon$ is arbitrary, the above measure must be $0$, and as $G$ is countable it follows that $\eta$ is supported on normalish subgroups of $G$.
	
	\begin{claim}
		The measure $\eta$ is not supported in the FC-center $\mathrm{FC}(G)$.
	\end{claim}
	\begin{proof}[Proof of the claim]
		Consider the closed subspace $\mathrm{Sub}(\mathrm{FC}(G))$, and notice that the $G$-orbit of any open set of the form 
		\[
		\{H \in \mathrm{Sub}(\mathrm{FC}(G)) : Q \cap H = Q \cap K\}, \, Q \subseteq \mathrm{FC}(G) \text{ finite and }K \in \mathrm{Sub}(\mathrm{FC}(G))
		\]
		is finite by the definition of $\mathrm{FC}(G)$. Since these sets constitute a basis of the topology of $\mathrm{Sub}(\mathrm{FC}(G))$, we conclude that the action of $G$ on the clopen subsets of $\mathrm{Sub}(\mathrm{FC}(G))$ has only finite orbits. This implies that the action is equicontinuous: indeed, it is an odometer in the sense of \cite[Définition 2.1.3]{Cornulier2014}. By \cite[Fait 2.1.4 (iii)]{Cornulier2014} it follows that $\mathrm{Sub}(\mathrm{FC}(G))$ is a projective limit of actions of $G$ on finite sets, so it must be equicontinuous. Hence any $\mu$-stationary measure on $\mathrm{Sub}(\mathrm{FC}(G))$ is invariant by \cite[Theorem 7.4]{FurstenbergGlasner2010}. Thus $\eta$ cannot be supported on $\mathrm{FC}(G)$, since any $\mu$-boundary which is also an invariant measure is a Dirac mass.
	\end{proof}
	Thus there exists $h \in G$ with infinite conjugacy class such that 
	\[
	c \coloneqq \eta(\{H \in \mathrm{Sub}(G): h \in H\})>0.
	\]
	
	Denote by $\mathbf{bnd} \colon G^\N\to \sub{G}$ the boundary map associated with $\eta$. Given a trajectory $\mathbf{w}=(w_n)_{n \geq 0} \in G^{\N}$ of the $\mu$-random walk on $G$, we can write
	\[
	\eta(\{H \in \sub{G} : h \in w_nHw_n^{-1}\}) = (w_n)_* \eta(\{H \in \sub{G} : h \in H\})
	\]
	for every $n\in \N$. The fact that $(\sub{G}, \eta)$ is a $\mu$-boundary implies
	\[
	\eta(\{H \in \sub{G} : h \in w_nHw_n^{-1}\}) \xrightarrow[n\to\infty]{} \delta_{\mathbf{bnd}(\mathbf{w})}(\{H \in \sub{G} : h \in H \})
	\]
	for $\mathbb{P}_{\mu}$-almost every trajectory $\mathbf{w} \in G^\N$. In particular, we have
	\begin{equation}\label{eq:boundaryCvg}
		\eta\left(\{H \in \sub{G} : h \in w_nHw_n^{-1}\}\right) \xrightarrow[n\to\infty]{}  1
	\end{equation}
	for every $\mathbf{w}=(w_n)_{n \geq 0}\in E\coloneqq \mathbf{bnd}^{-1}\left(\{H \in \sub{G} : h \in H\}\right)$, where
	\[
	\eta(E)=\mathbb{P_{\mu}}(\mathbf{bnd}^{-1}(\{H \in \sub{G} : h \in H\})) = \eta (\{ H \in \sub{G} : h \in H\}) = c >0.
	\]
	
	Consider one such trajectory $(w_n)_{n \geq 0}\in E$. The set $\{w_n^{-1}hw_n : n \geq 0\}$ is almost surely infinite because $h$ has infinite conjugacy class.
	
	Let $j \geq 2$. Since $\mu$ is non-degenerate, the action $G \curvearrowright (\mathrm{Sub}(G), \eta)$ is non-singular. Thus we can choose $\delta_j > 0$ such that for every measurable subset $A \subseteq \mathrm{Sub}(G)$ with $\eta(A) \geq 1 - \delta_j$ we have
	\[
	(g_i)_\ast\eta(A) \geq 1 - \frac{\varepsilon}{m 2^j}
	\]
	for every $i = 1,\ldots, m$. Using the convergence from Equation \eqref{eq:boundaryCvg}, we inductively construct an increasing sequence $(n_k)_{k \geq 2} \subseteq \N$ as follows. Let $n_2 \in \N$ be such that
	\[
	\eta(\{H \in \mathrm{Sub}(G) : w_{n_2}^{-1}hw_{n_2} \in H\}) \geq 1 - \delta_2.
	\]
	If $n_2, \ldots, n_k$ are already defined, find $n_{k+1} > n_k$ such that
	\[
	w_{n_{k+1}}^{-1}hw_{n_{k+1}}\notin \left\{h_{w_{n_2}},\ldots, h_{w_{n_k}} \right\} \, \text{ and }\, \eta(\{H \in \mathrm{Sub}(G) : w_{n_{k+1}}^{-1}hw_{n_{k+1}} \in H\}) \geq 1 - \delta_{k+1}.
	\]
	
	Our choice of $\delta_k$ guarantees that
	\begin{equation}\label{eq:csinequality}
		\eta(\{H : w_{n_k}^{-1}h w_{n_k} \in g_iHg_i^{-1}\}) = (g_i)_\ast \eta(\{H : w_{n_k}^{-1}hw_{n_k} \in H\}) \geq 1 - \frac{\varepsilon}{m 2^k}
	\end{equation}
	holds for every $k \geq 2$ and $i = 1, \ldots, m$. A union bound then shows that
	\[
	\eta\left(\left\{H \in \mathrm{Sub}(G) : \  w_{n_k}^{-1}h w_{n_k} \not \in \bigcap_{i = 1}^m g_iHg_i^{-1}\right\}\right) \leq\frac{\varepsilon}{2^k}
	\]
	for all $k \geq 2$, and hence that
	\[
	\eta\left(\left\{H \in \mathrm{Sub}(G) : \  (w_{n_k}^{-1}hw_{n_k})_{k \geq 2} \not \subseteq \bigcap_{i = 1}^m g_iHg_i^{-1}\right\}\right) \leq \epsilon,
	\]
	as desired, concluding the proof.
\end{proof}

\subsection{Baumslag-Solitar groups and normalish subgroups}
The Baumslag-Solitar groups are a family of one-relator groups
\[
\mathrm{BS}(m,n) = \langle a,t \mid ta^mt^{-1} = a^n\rangle,\, m,n \in \Z \setminus \{0\}
\]
introduced by Baumslag-Solitar in \cite{BaumslagSolitar1962} that has been well-studied from geometric and algebraic perspectives \cite{Whyte2001, levittAutomorphisms}. They are known to be C*-simple \cite[Proposition 5]{deLaHarpePreaux2011} whenever $m,n$ are both different from $\pm 1$ and $m \neq \pm n$. The infinite cyclic subgroup $\langle a \rangle$ is always normalish since $\langle a \rangle \cap g \langle a \rangle g^{-1}$ has finite index in $\langle a \rangle$ for all $g \in \mathrm{BS}(m,n)$ (see \cite[Section 2]{levittAutomorphisms} for instance).

Proposition \ref{prop:countableSubam} below finds the parameters $m,n$ such that $\mathrm{Sub}_\mathrm{am}(\mathrm{BS}(m,n))$ is countable. In particular, we deduce the following.

\begin{cor} \label{cor:noSRSs}
	The group $\mathrm{BS}(2,3)$ is C*-simple, admits an amenable normalish subgroup and all its amenable $\mu$-SRSs are trivial for every non-degenerate $\mu \in \mathrm{Prob}(\mathrm{BS}(2,3))$.
\end{cor}
\begin{proof}
	The first two statements are already known from \cite[Proposition 5]{deLaHarpePreaux2011} and \cite[Section 2]{levittAutomorphisms}, and Proposition \ref{prop:countableSubam} below shows that $\mathrm{BS}(2,3)$ has countably many amenable subgroups. Let $\mu$ be a non-degenerate probability measure on $G = \mathrm{BS}(2,3)$, and let $\eta$ be an ergodic $\mu$-SRS on $\subam{G}$. Since $\eta$ is supported on a countable set it must have an atom, and by ergodicity $\eta$ must be supported in a finite orbit of some $H \in \subam{G}$. If $H$ were non-trivial, then the normal subgroup generated by $H$ would be a non-trivial normal amenable subgroup of $G$, contradicting the C*-simplicity of $\mathrm{BS}(2,3)$. Thus, $H$ is trivial and so is $\eta$. The case when $\eta$ is non-ergodic follows from ergodic decomposition.
\end{proof}

\begin{prop} \label{prop:countableSubam}
	Let $m,n \in \Z \setminus \{0\}$ and $G = \mathrm{BS}(m,n)$. Then $G$ has only countably many amenable subgroups if and only if either
	\begin{itemize}
		\item $n$ or $m$ is equal to 1 or $-1$, or
		\item  $n = m$ or $n = -m$, or 
		\item $m/n$ and $n/m$ are not integers.
	\end{itemize}
\end{prop}
The proof of Proposition \ref{prop:countableSubam} uses a result of Y. Cornulier \cite{CornulierMSE2021} which states that, whenever neither $m/n$ nor $n/m$ are integers, the group $\mathrm{BS}(m,n)$ has no non-trivial element admitting infinitely many roots. We replicate the proof of this statement below for the convenience of the reader.

In the proof of Proposition \ref{prop:countableSubam} we will use the action of $\mathrm{BS}(m,n)$ on the \emph{Bass-Serre tree $\cT$ associated with its decomposition as an HNN-extension}. For our purposes, the tree $\cT$ is the tree with vertex set $\mathrm{BS}(m,n)/\langle a \rangle$ and oriented edges connecting $w \langle a \rangle$ with $wa^lt \langle a \rangle$ for all $w \in \mathrm{BS}(m,n)$ and $0 \leq l \leq n-1$. Left multiplication induces an action of $\mathrm{BS}(m,n)$ on $\cT$ through directed graph automorphisms. Bass-Serre theory ensures that $\cT$ is indeed a tree and that the action of $\mathrm{BS}(m,n)$ is transitive on oriented edges and vertices. Thus the stabilizers of vertices (resp. edges) are conjugates of $\langle a \rangle$ (resp. $\langle a^m \rangle$). A non-trivial element of $\mathrm{BS}(m,n)$ stabilizing no vertex on $\cT$ is said to be \emph{loxodromic}, and is said to be \emph{elliptic} otherwise. We refer to \cite[Chapitre I]{Serre1977} as the standard reference in Bass-Serre theory, and to \cite[Sections 1 \& 2]{CarderiGaboriauLeMaitreStalder2025} for Baumslag-Solitar groups and their Bass-Serre trees.

Let us recall the following lemma, which is a small variation on \cite[Proposition2.4]{CarderiGaboriauLeMaitreStalder2025}.
\begin{lema} \label{lema:countablesub}
	If $G$ is a countable group with a normal subgroup $N \subseteq G$ such that $N$ has countably many amenable subgroups and every subgroup of $G/N$ is finitely generated, then $G$ has countably many subgroups.
\end{lema}
\begin{proof}
	Write $\pi \colon G \to G/N$ the projection. A subgroup $H$ of $G$ is determined by $H \cap N$ and any section of a finite generating set of $\pi(H)$. There are always at most countably many options for these sections, and  if $H$ is amenable then there are at most countably many options for $H \cap N$.
\end{proof}

\begin{proof}[Proof of Proposition \ref{prop:countableSubam}]
	We first prove the forward implication. If $G = \mathrm{BS}(\pm 1,m)$ then $G$ has countably many amenable subgroups by \cite[Corollary 8.4]{BeckerLubotzkyThom2019}. If $n = \pm m$, then $G$ is virtually $F_n \times \Z$. The group $F_n \times \Z$ has countably many amenable subgroups by Lemma \ref{lema:countablesub}, and the same is true for $G$ by Lemma \ref{lema:countablesub} again.
	
	Now suppose that $m/n$ and $n/m$ are not integers. Let $\cT$ be the Bass-Serre tree associated to the HNN-extension decomposition of $G$ as $G = \langle a,t \mid ta^m t^{-1} = a^{n}\rangle$ and let $H \subseteq G$ be an amenable subgroup. Since $H$ has no non-abelian free subgroups and the action of $G$ on $\cT$ has no inversions, by J. Tits' categorization of actions of groups on trees \cite{Tits1970} either $H$ fixes a vertex, stabilizes the axis of a loxodromic element in $H$ or fixes an end of $\partial \cT$. In the first case $H$ is cyclic. In the second case, the stabilizer of the axis $\ell$ of a loxodromic element in $G$ splits as
	\[
	1 \longrightarrow K \longrightarrow \mathrm{Stab}_G(\ell) \longrightarrow D_\infty \longrightarrow 1
	\]
	where $D_\infty$ is the infinite dihedral group and the kernel $K$ is a subgroup of a vertex stabilizer (and hence either trivial or infinite cyclic). In any case, by Lemma \ref{lema:countablesub} the countably many subgroups $\mathrm{Stab}_G(\ell)$ where $\ell$ runs through the axes of loxodromic elements of $G$ have countably many subgroups.
	
	We are reduced to the case when $H$ fixes an end $\xi \in \partial \cT$. Fix a one-sided geodesic $(\xi_n)_{n \geq 0}$ representing $\xi$, and define $K \subseteq G$ as the subgroup of elliptic elements fixing $\xi$. Then 
	\[
	K = \bigcup_{m \geq 0} \mathrm{Fix}_G( \xi_{[m, \infty)} )
	\]
	is the ascending union of the pointwise fixators of the rays $\xi_{[m, \infty)} = (\xi_{m + n})_{n \geq 0},\, m \geq 0$. Thus $K$ is either trivial, infinite cyclic or a direct union of infinite cyclic groups. In the last case, there is an element of $K \setminus \{e_G\}$ with infinitely many roots.
	
	\begin{claim}[\cite{CornulierMSE2021}]
		If $m/n$ and $n/m$ are not integers, then $G = \mathrm{BS}(m,n)$ contains no non-trivial element with roots of arbitrarily large order.
	\end{claim}
	\begin{proof}[Proof of the claim]
		Let $h \colon G/ \langle a \rangle \to \Z$ the \emph{height function} associating to each coset $g \langle a \rangle$ the signed number of $t$'s appearing in $g$. The function $h$ is well defined (it factors through the quotient of $G$ by the normal subgroup generated by $a$).
		
		Let $g \in G \setminus \{e_G\}$ that fixes a vertex $v = w \langle a \rangle$ in $\cT$, so $g = wa^{\zeta_g(v)} w^{-1}$ for a unique $\zeta_g(v) \in \Z \setminus \{0\}$. If $g$ also fixes a vertex $v' = wa^lt \langle a \rangle$ in $\cT$ that is adjacent to $v$ (that is, there is an oriented edge from $v$ to $v'$ in $\cT$), the equality
		\[
		wa^{\zeta_g(v)} w^{-1} = wa^l t a^{\zeta_g(v')} (wa^l t)^{-1}
		\]
		implies that $\zeta_g(v') = (m/n)\zeta_g(v)$. By induction we see that for any pair of vertices $v,v'$ in the subtree $\cT_g$ of fixed points of $g$ in $\cT$ we have $\zeta_g(v') = (m/n)^{h(v') - h(v)} \zeta_g(v)$. Since no non-trivial power of $m/n$ is an integer and the $\zeta_g(v)$ are always integers, we deduce that $h$ and $\zeta_g$ are bounded on $\cT_g$.
		
		Now assume that $g \in G \setminus \{e_G\}$ has infinitely many roots, and let $M > 0$ be an upper bound for $\{\zeta_g(v) : v \in G/\langle a \rangle\}$. The action of $g$ on $\cT$ cannot be loxodromic, since the roots of $g$ would have arbitrarily small translation length. Thus $g$ and any of its roots are elliptic. Now let $u \in G$ with $u^N = g$ for some $N \in \N_+$, and let $v'$ be a vertex of $\cT$ fixed by $u$. Then
		\[
		N \leq N \abs{\zeta_u(v')} = \abs{\zeta_g(v')} \leq M. \qedhere
		\]
	\end{proof}
	
	Thus $K$ must be cyclic. If $\xi$ is fixed by no loxodromic element of $G$ then $H \subseteq \mathrm{Stab}_G(\xi) = K$, so there are countably many options for $H$.
	
	If not, let $g \in G$ be a loxodromic element with minimal translation length $d \in \N_+$. Then $\mathrm{Stab}_G(\xi) = K \rtimes \Z$, where the generator of the right-hand side is $g$ and conjugation by $g$ sends every $\mathrm{Fix}_G(\xi_{[m, \infty)} )$ to $\mathrm{Fix}_G(\xi_{[m + d, \infty)})$. Again, Lemma \ref{lema:countablesub} shows that $\mathrm{Stab}_G(\xi)$ has countably many subgroups, and since there are countably many fixed ends in $\partial \cT$ by loxodromic elements of $G$ we conclude that $H$ varies in a countable set. This finishes the proof that $\mathrm{BS}(m,n)$ has countably many amenable subgroups.
	
	For the reverse implication, consider $m,n \in \Z \setminus \{0\}$ not satisfying any of the conditions in the statement of the proposition. Since $\mathrm{BS}(m, n)$ is isomorphic to $\mathrm{BS}(-m,-n)$ we may assume that $m > 0$, and since $\mathrm{BS}(m,n)$ contains $\mathrm{BS}(m^2, n^2)$ (see \cite[Theorem 1.3]{Levitt2015} for instance) we may assume that $n > 0$ too. Now $m \geq 2$ and we may write $n = km$ with $k \geq 2$.
	
	As before, let $\cT$ be the Bass-Serre tree associated to the HNN-extension decomposition of $G$. Let $\widetilde{\cT}$ be the subtree of $\cT$ with vertex set $\{wt^{-1} \langle a \rangle : w \text{ is a finite word in }\{a, t^{-1}\} \}$. The tree $\widetilde{\cT}$ is a rooted complete $m$-ary tree with root $\langle a \rangle$.
	
	\begin{claim}
		The fixed points of $a^m$ in $\cT$ are the vertices of $\widetilde{\cT}$.
	\end{claim}
	\begin{proof}[Proof of the claim]
		The element $a^m$ fixes $\langle a \rangle$ and does not fix any of its neighbors
		\[
		t\langle a \rangle, at \langle a \rangle, a^2t\langle a \rangle, \cdots, a^{km-1}t \langle a \rangle
		\]
		because $k \geq 2$, so the subtree $\cT_{a^m}$ of fixed points of $a^m$ is contained in $\widetilde{\cT}$. Now let $wt^{-1} \langle a \rangle$ be a vertex of $\widetilde{\cT}$, and write $w = a^{n_1}t^{-1} a^{n_2} t^{-1} \cdots a^{n_l}$ for some $l \in \N_+$ and $n_1, \ldots, n_l \in \N$, so
		\begin{align*}
			a^mw t^{-1} & = a^m a^{n_1}t^{-1} a^{n_2} t^{-1} \cdots a^{n_l}t^{-1} = a^{n_1} t^{-1}a^{km} a^{n_2}t^{-1} \cdots a^{n_l}t^{-1} \\
			& = a^{n_1} t^{-1}a^{n_2}t^{-1}a^{k^2m} \cdots a^{n_l}t^{-1} = \cdots = a^{n_1} t^{-1}a^{n_2}t^{-1} \cdots a^{n_l}t^{-1}a^{k^l m}
		\end{align*}
		and thus $a^m(w t^{-1} \langle a \rangle) = w t^{-1} \langle a \rangle$. We conclude that $\cT_{a^m} = \widetilde{\cT}$.
	\end{proof}
	
	Let $\xi_1, \xi_2$ be distinct ends of $\widetilde{\cT}$ and identify them with geodesics in $\widetilde{\cT}$ starting at $\langle a \rangle$. Let $gt^{-1} \langle a \rangle$ be a vertex of $\xi_1 \setminus \xi_2$, so $(gt^{-1}) a^m(gt^{-1})^{-1}$ fixes $\xi_1$ and does not fix $\xi_2$. Thus the many amenable subgroups $\mathrm{Stab}_G(\xi), \xi \in \partial \widetilde{\cT}$ are pairwise distinct, and there are uncountably many subgroups of this type since $m \geq 2$. This finishes the proof of the proposition.
\end{proof}

\subsection{Property (CS)}
A countable group $G$ is said to have \emph{property (CS)} if, for every unitary representation $\pi \colon G \to B(\mathcal{H})$ that is weakly contained in the left-regular representation $\lambda$ of $G$, there exists an neighborhood $U$ of $\mathrm{id}_{\mathcal{H}}$ in the strong operator topology of $B(\mathcal{H})$ such that $\pi^{-1}(U)$ is contained in the amenable radical $R_a(G)$. Recall that $\pi$ is \emph{weakly contained in $\lambda$} if $\norm{\pi(f)} \leq \norm{\lambda(f)}$ for all $f \in \C[G].$ This notion was introduced in \cite{BreuillardKalantarKennedyOzawa2017} as a sufficient condition for a group $G$ to satisfy a conjecture of Connes-Sullivan on subgroups of connected Lie groups acting amenably on homogeneous spaces, proved by R. Zimmer \cite{Zimmer1978}. Linear groups and discrete groups $G$ with $H^2_\mathrm{b}(G, \ell^2(G/R_a(G)) \neq 0$ are known to have property (CS) \cite[Theorems 8.4 \& 8.6]{BreuillardKalantarKennedyOzawa2017}.

Property (CS) is stronger than C*-simplicity: every group with property (CS) and trivial amenable radical is C*-simple \cite[Proposition 8.2]{BreuillardKalantarKennedyOzawa2017}. Indeed, a discrete group with no non-trivial finite normal subgroups and no amenable normalish subgroups is C*-simple \cite[Theorem 6.2]{BreuillardKalantarKennedyOzawa2017}, while property (CS) implies the absence of amenable normalish subgroups. To see this, it suffices to note that if a subgroup $H \subseteq G$ is amenable and normalish, then the quasi-regular representation $\lambda_{G/H}$ is weakly contained in the left regular representation, yet the image of $G$ under $\lambda_{G/H}$ is not discrete in the strong operator topology. Combined with Theorem \ref{thm: normalish}, this yields the following.

\begin{cor}
	Let $G$ be a countable group with finite amenable radical and suppose that there exists a non-degenerate $\mu \in \mathrm{Prob}(G)$ which admits an amenable boundary $\mu$-SRS distinct from a Dirac mass on a finite normal subgroup of $G$. Then $G$ does not have property (CS).
\end{cor}

In fact, the existence of an amenable boundary $\mu$-SRS $\eta$ allows us to explicitly construct a unitary representation of $G$ that witnesses the failure of property (CS). Following the notation of \cite[Section F.5]{BekkaDeLaHarpe2020}, let us define the direct integral $\mathcal{H}_{\Pi}$ of $(\ell^2(G/H))_{H \in \mathrm{Sub}_\mathrm{am}(G)}$ as the Hilbert space \[\mathcal{H}_\Pi = \int_{\mathrm{Sub}_\mathrm{am}(G)} \ell^2(G/H) \dd\eta(H),\] that is, $\mathcal{H}_\Pi$ consists of all maps $H \in \mathrm{Sub}_\mathrm{am}(G) \mapsto v_H \in \ell^2(G/H)$ such that \[\int_{\mathrm{Sub}_\mathrm{am}(G)} \norm{v_H}^2_{\ell^2(G/H)} \dd \eta(H) < \infty\] and such that $H \in \mathrm{Sub}_\mathrm{am}(G) \mapsto (v_H)_{gH}$ is measurable for every $g \in G$. Then, $\mathcal{H}_\Pi$ is a Hilbert space with inner product \[ \langle v, w \rangle = \int_{\mathrm{Sub}_\mathrm{am}(G)} \langle v_H, w_H \rangle_{\ell^2(G/H)} \dd\eta(H)\] for every $v = (v_H)_H$, $w = (w_H)_H \in \mathcal{H}_\Pi$. 

\begin{defn}\label{def: direct integral representations}
	We define the direct integral of the family of all quasi-regular representations $(\lambda_{G/H})_{H \in \mathrm{Sub}_\mathrm{am}(G)}$ as the unitary representation $\Pi : G \to B(\mathcal{H}_\Pi)$ defined by $(\Pi(g)v)_H = \lambda_{G/H}(g)v_H$ for every $g \in G$, $v \in \mathcal{H}_\Pi$ and $H \in \mathrm{Sub}_\mathrm{am}(G)$. Define $\pi$ to be the cyclic subrepresentation of $\Pi$ generated by the vector $v = (\delta_H)_{H \in \mathrm{Sub}_\mathrm{am}(G)} \in \mathcal{H}_\Pi$.
\end{defn}

Every quasi-regular representation $\lambda_{G/H}$ is weakly contained in the left-regular representation $\lambda$, since $H \in \mathrm{Sub}_\mathrm{am}(G)$. Therefore, $\pi$ is weakly contained in $\lambda$. The next result shows that such representation provides a concrete witness showing that the group $G$ cannot satisfy property (CS).

\begin{prop}\label{prop: explicit CS}
	Let $G$ be a countable group with finite amenable radical and suppose that there exists a non-degenerate $\mu \in \mathrm{Prob}(G)$ which admits an amenable boundary $\mu$-SRS distinct from a Dirac mass on a finite normal subgroup of $G$. 
	
	Let $\pi$ be the cyclic subrepresentation of the direct integral representation $\Pi$ generated by the vector $v = (\delta_H)_{H \in \mathrm{Sub}_\mathrm{am}(G)}$ (Definition \ref{def: direct integral representations}). Then, the identity operator in $B(\mathcal{H}_{\pi})$ is not isolated in the strong operator topology.
\end{prop}
\begin{proof}
	We show that, for every SOT-neighborhood $U$ of the identity in $B(\mathcal{H}_{\pi})$, there exists an element $g \in G \setminus \{e_G\}$ such that $\pi(g) \in U$.  Since the representation $\pi$ is generated by $v$, it suffices to verify this for neighborhoods of the form
	\[
	U = \{ T \in B(\mathcal{H}_{\pi}) : \norm{\pi(g_i^{-1}) T \pi(g_i) v - v}_{\mathcal{H}_{\pi}} < \varepsilon \text{ for all } i=1, \dots, m\}
	\]
	for every $\varepsilon >0$ and $g_1, \dots, g_m \in G$.
	
	Now,one can follow the proof of Theorem \ref{thm: normalish} for $j=2$. Indeed, the convergence from Equation \eqref{eq:boundaryCvg} and the non-singularity of the action $G \curvearrowright \mathrm{Sub}_\mathrm{am}(G)$ together imply that there exists $k \in \N$ and $(w_n)_{n \geq 0} \in G^\N$ such that $\eta(\{ H \in \mathrm{Sub}_\mathrm{am}(G) : w_k h w_k^{-1} \in H\}) \geq 1-\delta$, where $\delta > 0$ is chosen in such a way that
	\begin{align*}
		(g_i)_* \eta(\{ H \in \mathrm{Sub}_\mathrm{am}(G) : w_k^{-1}h w_k \in H \}) & = \eta(\{ H \in \mathrm{Sub}_\mathrm{am}(G) : w_k^{-1} h w_k \in g_i H g_i^{-1}\}) \\ & \geq 1-\varepsilon^2 / 2m
	\end{align*}
	for every $i=1, \ldots, m$. Set $g \coloneqq w_k^{-1} h w_k \in G \setminus \{e_G\}$. A union bound then shows that
	\[
	\eta(\{ H \in \mathrm{Sub}_\mathrm{am}(G) : g_i^{-1}gg_i \in H\}) > 1-\varepsilon^2 / 2.
	\]
	We obtain
	\begin{align*}
		\norm{\pi(g_i^{-1}gg_i)v-v}^2_{\mathcal{H}_\pi} &= \int_{\mathrm{Sub}_\mathrm{am}(G)} \norm{\lambda_{G/H}(g_i^{-1} g g_i) \delta_H - \delta_H}^2_{\mathcal{H}_{\pi}} \dd\eta(H) \\
		&= 2 \eta(\{ H \mid g_i^{-1}gg_i \notin H \}) \leq \varepsilon^2
	\end{align*}
	for every $i=1, \dots, m$, which finally shows that $\pi(g) \in U$. 
\end{proof}

	\bibliographystyle{alpha}
	\bibliography{newbiblio}
\end{document}